\documentclass[11pt]{amsart}
\usepackage{amsmath}
\usepackage{amsthm,amssymb}
\usepackage{hyperref}
\usepackage{amssymb}
\usepackage{graphicx}
\usepackage[utf8]{inputenc}
\input xy
\xyoption{all}
\usepackage{color}
\usepackage{enumerate}
\newtheorem{theorem}{Theorem}[section]

\newtheorem{lemma}[theorem]{Lemma}

\newtheorem{corollary}[theorem]{Corollary}
\newtheorem{proposition}[theorem]{Proposition}
\theoremstyle{definition}
\newtheorem{definition}[theorem]{Definition}
\newtheorem{remark}[theorem]{Remark}
\newtheorem{problem}[theorem]{Problem}

\newcommand{\End}{\mbox{\rm End}}

\newcommand{\diag}{\mbox{\rm diag}}

\newcommand{\ch}{\mbox{\rm char}}

\newcommand{\N}{\mathbb{N}}
\newcommand{\Z}{\mathbb{Z}}
\newcommand{\C}{\mathbb{C}}

\begin{document}

\title[Rings generated by periodic elements]{Rings additively generated by \\ periodic elements}

\author[M. H. Bien]{M. H. Bien$^{1,2}$}
\author[P. V. Danchev]{P. V. Danchev$^{3}$}
\author[M. Ramezan-Nassab]{M. Ramezan-Nassab$^{4,5}$}

\address{[1] Faculty of Mathematics and Computer Science, University of Science, Ho Chi Minh City, Vietnam}

\address{[2] Vietnam National University, Ho Chi Minh City, Vietnam}

\address{[3] Institute of Mathematics and Informatics, Bulgarian Academy of Sciences, 1113 Sofia, Bulgaria}

\address{[4] Department of Mathematics, Kharazmi University, 50 Taleghani Street, Tehran, Iran} ~~
    \address{[5] School of Mathematics, Institute for Research in Fundamental Sciences (IPM), P.O. Box 19395-5746, Tehran, Iran.} ~~~~ \\ \\ \\

    \medskip

\email{M. H. Bien: mhbien@hcmus.edu.vn \newline
        P. V. Danchev: danchev@math.bas.bg; pvdanchev@yahoo.com \newline
        M. Ramezan-Nassab: ramezann@khu.ac.ir}

\keywords{Division rings; Torsion units; Periodic elements and rings.\\
\protect
\indent{\it Corresponding Author: Peter V. Danchev}}

\begin{abstract}
In the present paper, as a generalization of the classical periodic rings, we explore those rings whose elements are additively generated by two (or more) periodic elements by calling them {\it additively periodic}. We prove that, in some major cases, additively periodic rings remain periodic too; this includes, for instance, algebraic algebras, group rings, and matrix rings over commutative rings. Moreover, we also obtain some independent results for the new class of rings; for example, the triangular matrix rings retain that property.
\end{abstract}

\maketitle


\section{Introduction and Known Facts}

Everywhere in the text of this paper, all rings are assumed to be associative with an identity. Recall that an element $r$ of a ring $R$ is said to be {\it periodic} if $r^n = r^m$ for some natural numbers $n$ and $m$, where $n > m \geq 1$. In particular, when $m = 1$, the element $r$ is called a  {\it potent} or, more precisely, an {\it $n$-potent}. A ring $R$ is referred to as a {\it periodic ring} if all of its elements are periodic. Furthermore, $R$ is said to be {\it weakly periodic} if every element $x$ in $R$ can be expressed as $x = a + b$, where $a$ is a potent element and $b$ is a nilpotent element. It is known that every periodic ring is weakly periodic, but according to \cite[Examples 3.1 and 3.2]{Ser}, the converse is not true. There exist numerous papers in the literature that discuss (weakly) periodic rings, including references such as \cite{ABD, Cui, Di, Hirano}.

\medskip

Now, let us define a broader class of rings than that of weakly periodic.
\begin{definition}\label{def}
Let $R$ be a ring and let $k$ be a natural number. We say that  $a\in R$ is an {\it additively $k$-periodic element} if $a$ can be written as a sum of at most $k$ periodic elements in $R$. The ring $R$ is called an {\it additively $k$-periodic ring} if each of its elements is additively $k$-periodic. Finally, $R$ is called {\it additively periodic} if each of its element $a\in R$ is additively $m$-periodic for some $m=m(a)$.
\end{definition}

It is clear that every additively $1$-periodic ring is periodic, and that every weakly periodic ring is additively $2$-periodic. However, as we will see in Remark~\ref{rem}, there are additively periodic rings that are {\it not} weakly periodic. Thus, we have the following inclusion's relationships between these classes of rings:

$$\{\text{periodic} \}\subsetneqq \{ \text{weakly periodic}\}\subsetneqq\{\text{additively periodic}\}$$
and
$$\{ \text{weakly periodic}\}\subseteq \{ \text{additively $2$-periodic}\}\subseteq
 \{\text{additively periodic}\}.$$

While we are still uncertain about whether or not there exists an additively periodic ring that is not additively $2$-periodic, and whether or not there exists an additively $2$-periodic ring that is not weakly periodic, this paper aims to investigate which of these rings are actually still periodic.

\medskip

Our main motivation for Definition~\ref{def} and for the current study comes from the following three points of view:

\medskip

Firstly, in many papers, such as \cite{ACDT, AT, Ser, YKZ}, the authors studied the rings in which each element is the sum of two idempotents, one idempotent and one tripotent, one idempotent and one nilpotent, and finally, two tripotents. However, it is worth noting that all of these elements commute pairwise.

Our second motivation for the current study is concerned with the Diesl's recent work in~\cite{Di}, where he considered rings whose elements are sums of finitely many potents and one nilpotent element, all of which commute pairwise. Precisely, it was established there that interesting decomposable results exist for such rings, thus somewhat simplifying their complicated structure. Generally, our focus of interest is to investigate rings, which are possibly non-commutative, in which each element is additively generated by periodic elements. In fact, in contrast to \cite{Di}, we do not require our elements to commute with each other.

And the final motivation of our work deals with the structure of division rings. In fact, the famous theorem of Wedderburn asserts that any finite division ring is always a field (see, for instance, \cite{Lam}).  Moreover, by the remarkable well-known Kaplansky's result, if a division ring $D$ is radical over its center $F$ (i.e., for any $x\in D$, there is $n\in \mathbb{N}$, depending on $x$, such that $x^n\in F$), then $D$ is a field (see, e.g., \cite[Theorem 15.15]{Lam}). Therefore, in addition, if $D$ is a division ring with torsion unit group $\mathcal{U}(D)$, then $D$ is a field. As every periodic unit element is, in fact, a torsion unit, this implies that every periodic division ring is a field. Therefore, we naturally arrive at the following expansion of the last fact, which poses the following challenging question:

\begin{problem}\label{major}
If in a division ring each element is a finite sum of torsion units, is then this ring a field?
\end{problem}

Since this question seems to be extremely insurmountable at this stage, in what follows, we shall partially resolve this problem in the affirmative for some concrete cases by using some non-standard results from field theory.

\medskip

Concretely, our work is organized as follows: in Section 2, we study those rings whose elements are (additively) generated by periodic elements as we distribute our results into four subsections, bearing in mind their scientific directions. Indeed, in the first subsection, we mainly concentrate on some basic facts concerning additively periodic rings. As a remarkable result, we show that each commutative additively periodic as well as any algebraic additively $2$-periodic ring is periodic (Theorems~\ref{commuting} and \ref{algebraic}). Further, in the next second subsection, we are pertained to the examination of group rings as we achieve Theorem~\ref{nilpotent} which shows that if $R$ is a commutative ring and $G$ is a nilpotent group, then the group ring $RG$ is additively periodic if, and only if, $RG$ is periodic. In the  third subsection, we are involved to examine triangular and full matrix rings and succeed to prove Theorem~\ref{triang}. Particularly, the triangular matrix ring $\mathbb{T}_n(R)$ (resp., the full matric ring ${\rm M}_n(R)$) is additively periodic (resp., additively $2$-periodic) if, and only if, $\mathbb{T}_n(R)$ (resp., ${\rm M}_n(R)$) is periodic. In the final fourth subsection, we are devoted to the so-called {\it torsion product property} and here we obtain as a main result Theorem~\ref{strongly}. As a mentioned result, in Corollary~\ref{2-torsion}, we will show that if $\mathcal{U}(R)$ is either a torsion or a locally nilpotent group and, in both cases, each element of $R$ is a sum of two unit elements, then $R$ is periodic.

In closing Section~3, we give some closely related concluding remarks that, hopefully, will stimulate a further intensive research of the topic, as well as we also state four relevantly difficult and still unsettled questions of some interest and importance.

\medskip


\medskip

\section{Rings Generated by Periodic Elements}

\subsection{Basic facts and general results}
In this subsection, we provide the main basic result that we will use throughout our paper.

For any ring $R$, we denote its multiplicative group by ${\mathcal U}(R)$. Note that if $a\in {\mathcal U}(R)$ is periodic, then $a$ is, in fact, a {\it torsion unit} or just a {\it root of unity}, as $a^k=1$ for some $k\in \mathbb{N}$.

\medskip

The next simple claim is the key in proving our chief result listed below.


\begin{lemma}\label{fields}
For a field $F$, the next two points are equivalent$:$
\begin{itemize}
\item[(1)] $F$ is an algebraic extension of a finite field.
\item[(2)] $F$ is an additively periodic ring.
\end{itemize}
\end{lemma}

\begin{proof}

(1) $\Rightarrow$ (2). Let $F \supseteq K$ be a finite field, and let $F/K$ be an algebraic extension. Assume $a \in F$. Since $a$ is algebraic over $K$, we have that $K(a)/K$ is a finite extension. Hence, $K(a)$ is a finite field. Therefore, $a$ is a root of unity.

(2) $\Rightarrow$ (1). First, assume that $\text{char}(F)=0$, thus $\mathbb{Q}\subseteq F$. Let $a$ be any nonzero element of $\mathbb{Q}$. By hypothesis, one can write that $$a=a_1+\cdots+a_k,$$ where for each $1\leq i\leq k$, $a_i\in F$ is a root of unity. Thus, the subring $\mathbb{Z}[a_1,\dots,a_k]$ of $F$ generated by $a_1,\dots,a_k$ over $\mathbb{Z}$ is integral over $\mathbb{Z}$ (see \cite[Proposition~5.1 and Corollary~5.3]{Atiyah}). Consequently, $a\in\mathbb{Z}[a_1,\dots,a_k]$ is integral over $\mathbb{Z}$. Since $a$ was arbitrary from $\mathbb{Q}$, this contradicts the fact that $\mathbb{Z}$ is integrally closed (see \cite[Example~5.0]{Atiyah}).

Therefore, we may assume that $\text{char}(F)=p>0$. Then, $\mathbb{F}_p$, the finite field of $p$ elements, is contained in $F$, and clearly, $F/\mathbb{F}_p$ is an algebraic extension, as required.
\end{proof}

\medskip

The next two technicalities are very useful for establishing our further results.


\begin{lemma}\label{char}
If  $R$ is an additively $2$-periodic ring, then $R$ is of positive characteristic.
\end{lemma}

\begin{proof}
If $R$ is of zero characteristic, then $\Z$ can be viewed as a subring of $R$. By hypothesis, there are two periodic elements $u$ and $v$ in $R$ such that $3=u+v$, and, as $u+v$ is central in $R$, we have $uv=vu$. Thus, the subring $B=\Z[u,v]$ of $R$ generated by $u$ and $v$ over $\Z$ is necessarily a commutative ring which is integral over $\Z$.

Now, let $f:\Z \to \C$ be the inclusion homomorphism. So, consulting with \cite[Exercise~2, p.~67]{Atiyah}, there exists a ring homomorphism $\alpha:B\to \C$ such that $\alpha|_\Z=f$. Therefore, $3=\alpha(3)=\alpha(u)+\alpha(v)$. Note that, if $x$ is a periodic element of $\C$, then the absolute value $|x|$ is either $0$ or $1$. Consequently,
$$3=|\alpha(u)+\alpha(v)|\leq |\alpha(u)|+|\alpha(v)|\leq 2,$$
which is a contradiction. Hence, the characteristic of $R$ is non-zero, as asserted.
\end{proof}


\begin{remark}\label{remark-k}
Let $R\subseteq S$ be two rings such that each  element of $R$ can be  expressed as a sum of $k$ commuting periodic elements of $S$ for some natural number $k$. A similar reasoning as in the proof of Lemma~\ref{char} shows that $R$ has a positive characteristic.
\end{remark}


\begin{lemma}\label{commute-periodic}
Let $R$ be a ring of positive characteristic, and let $k$ be a natural number. If $x\in R$ can be written as a sum of $k$ commuting periodic elements of $R$, then $x$ itself is periodic.
\end{lemma}

\begin{proof}
Write $x=x_1+x_2+\cdots+x_k$, where each $x_i$ is periodic and $x_ix_j=x_jx_i$ for all $1\leq i,j\leq k$. It is readily seen that it suffices to prove the statement only for the case $k=2$. To that aim, let $\ch (R)=n>0$. Clearly, the set
$$\{\alpha x_1^ix_2^j\mid \alpha\in\{0,1,\dots,n-1\}, i,j\in\N\}$$
is finite. As $x_1x_2=x_2 x_1$, this implies that the set $\{x^t\mid t\in \N\}$ is also finite. Therefore, $x$ is  periodic, as formulated.
\end{proof}


Combining the results of Remark~\ref{remark-k} and Lemma~\ref{commute-periodic}, we can easily deduce one of our following main results.

\begin{theorem}\label{commuting}
Let $R$ be a ring and $k$ a natural number. If every element of $R$ is expressed as a sum of $k$ commuting periodic elements, then $R$ is periodic. Particularly, each commutative additively periodic ring is periodic.
\end{theorem}


\begin{remark}\label{rem}
We certainly cannot remove the assumption of ``commuting" in Theorem~\ref{commuting}. To see this, let $V$ be an infinite-dimensional vector space over a field $F$. Let $p_i(t)$, $1\leq i\leq 4$, be split polynomials of degree $2$ in $F[t]$. Then, by \cite[Theorem~1.1]{Pazzis}, every element $\alpha\in\End_F(V)$ can be represented as $\alpha=\sum_{i=1}^4 \alpha_i$ where $\alpha_i\in \End_F(V)$ and $p_i(\alpha_i)=0$ for all $1\leq i\leq 4$.
Particularly, if we let $p_i(t)= t^2-t$ or $p_i(t)= t^2-1$, this implies that each element of $\End_F(V)$ is a sum of at most four periodic elements (idempotents or involutions). Now, if $F$ is of zero characteristic, then clearly $\End_F(V)$ is {\it not} (weakly) periodic, as wanted.

As another example, let $\mathcal H$ be an infinite-dimensional (complex) Hilbert space and denote $B(\mathcal H)$ as the ring of all bounded linear operators on $\mathcal H$. Then, according to \cite[Corollary~3.2]{Hilbert}, every operator $A\in B(\mathcal H)$ can be decomposed as a sum of at most four automorphisms of order 3. However, it is clear that $B(\mathcal H)$ is {\it not} (weakly) periodic, as desired.
\end{remark}


We will show in the next result that we can remove the assumption of ``commuting" in Theorem~\ref{commuting} for the algebraic algebra in the case when $k=2$.

\begin{theorem}\label{algebraic}
Let $F$ be a field and let $R$ be an algebraic $F$-algebra. Then, in each of the following cases, $R$ is a periodic ring:
\begin{itemize}
\item[(1)]   each element of $F$ is a sum of $k$ commuting periodic elements in $R$.
\item[(2)] $R$ is additively $2$-periodic.
\item[(3)]  $\mathcal{U}(R)$ is a torsion group.
\end{itemize}
\end{theorem}
\begin{proof}
To prove (1), by combining Remark~\ref{remark-k} and Lemma~\ref{commute-periodic}, we see that $F$ is periodic. Therefore, by Lemma~\ref{fields}, $F$ is algebraic over a finite field $\mathbb{F}_p$, and so is $R$ as well.  Now, for each $\alpha\in R$, $\mathbb F_p[\alpha]$, the subring of $R$ generated by $\alpha$ over $\mathbb F_p$, is finite. Therefore, $\alpha$ is a periodic element, showing that $R$ is periodic.

Suppose now that $R$ is additively $2$-periodic and choose $x \in F$.  Write $x = u + v$, where $u$ and $v$ are two periodic elements of $R$. Since $x$ is central, the equality $uv = vu$ holds. Then, the assertion~(1) shows that $R$ is periodic, which proves~(2).

Finally, assume that $\mathcal{U}(R)$ is a torsion group. If $F$ is finite, the reasoning is similar to the first part and $R$ is a periodic ring. Thus, we may assume that $F$ is infinite. Then, by \cite[Corollary~2.9]{2-good}, each element of $R$ is a sum of two elements in $\mathcal{U}(R)$, which are torsion by assumption, i.e., $R$ is additively $2$-periodic. Hence, the result follows from the previous part.
\end{proof}


Standardly, for any ring $R$, the symbol ${\rm M}_n(R)$ means the {\it full matrix ring} over $R$ of size $n\geq 1$. Let $D$ be a division ring that is algebraic over its center $F$, and let $n>1$ be any natural number. Although ${\rm M}_n(D)$ may not be algebraic over $F$ (see \cite[Theorem~8.4.1]{Bo_Co_97}), from Theorem~\ref{algebraic}, we can deduce the following result which answers partially Problem~\ref{major}.

\begin{corollary}\label{divisiontwo}
Let $D$ be a division ring which is algebraic over its center $F$. Assume  $n$ and $k$ are   two natural numbers.
\begin{itemize}
\item[(1)] If each element of $F{\rm I}_n$ is a sum of $k$ commuting periodic elements in ${\rm M}_n(D)$, then $D$ is a locally finite field, whence the ring ${\rm M}_n(D)$ is locally finite, too.
\item[(2)] If  ${\rm M}_n(D)$ is additively $2$-periodic, then both of the rings $D$ and  ${\rm M}_n(D)$ are locally finite.
\end{itemize}
\end{corollary}
\begin{proof}
By Theorem~\ref{algebraic}, $D$ is periodic. Therefore, $D$ is a field (that is, locally finite by Lemma~\ref{fields}). Referring to \cite[Corollary 2.3]{Kim}, we obtain that ${\rm M}_n(D)$ is a locally finite ring.
\end{proof}


For some special cases of additively periodic division algebras, we can only demonstrate that the characteristics of such algebras are positive: Following Amitsur \cite{Am}, an algebraic algebra $A$ over a field $F$ will be said to be of locally bounded degree (hereafter, abbreviated as a {\it LBD-algebra} for short) if every finitely generated submodule of $A$ consists of elements with bounded degrees. Apparently, every locally finite-dimensional algebra is an LBD-algebra. According to \cite[Theorem 5]{Am}, every algebraic algebra over an uncountable field is also an LBD-algebra.

\medskip

First, an useful simple lemma is needed.

\begin{lemma}\label{LBD}
Let $K$ be an infinite field and let $D$ be an LBD division $K$-algebra. Then
\begin{itemize}
\item [(i)] for any integer $m \ge 2$ and $m$-tuple $\theta _{1},
\dots , \theta _{m} \in D$, there exist $\tilde m \in \mathbb{N}$ and $\tilde \theta \in D$, such that $[K(\tilde \theta )\colon K] = \tilde m$, $\tilde m$ is divisible by $s_i:=[K(\theta _{i})\colon K]$, $i = 1, \dots , m$, and by $s'_m:=[K(\theta _{m}')\colon K]$, where $\theta _{m}' = \sum _{i=1} ^{m} \theta _{i}$; and
\item [(ii)] $\sum _{i=1} ^{m} {\tilde m}\cdot  s _{i} ^{-1}{\rm Tr}_{K(\theta _{i})/K}(\theta _{i}) = \tilde m\cdot s _{m} ^{\prime -1}{\rm Tr}_{K(\theta _{m}')/K}(\theta _{m}')$.
\end{itemize}
\end{lemma}

\begin{proof}
When $m = 2$, this is contained in \cite[Proposition~3.2]{Chip}; and, in general, the assertions are proved by proceeding standardly with induction on $m$ and utilizing \cite[Proposition~3.1]{Chip}.
\end{proof}


We, thereby, arrive at our next technical claim.

\begin{proposition}\label{FC-field}
Let $K$ be either a local or a global field (more generally, an FC-field in the sense of \cite{Chip}), and let $D$ be an LBD division $K$-algebra. If  $D$ is additively periodic, then $D$ is of positive characteristic.
\end{proposition}

\begin{proof}
Suppose on the contrary that $\ch(K) = 0$ and $D$ satisfies the stated condition. Consulting with \cite[Lemma~3.9]{Chip}, for each prime number $p$, there is $k(p) \ge 0$ such that $p ^{k(p)+1}$ does not divide the degree $[K(\rho )\colon K]$ for any $\rho \in D$. Considering now a presentation of $p ^{-k(p)-1}= \sum _{i=1} ^{m} \theta _{i}$ as a sum of torsion units $\theta_i\in D$ for a fixed $p$. Evidently, $m\geq 2$. Then, one obtains from condition~(ii) of Lemma~\ref{LBD} that
\begin{equation}\label{eq11}
\sum _{i=1} ^{m} {\tilde m}\cdot  s _{i} ^{-1}{\rm Tr}_{K(\theta _{i})/K}(\theta _{i}) = \mu _{p}/p,
\end{equation}
where $\mu _{p}=\tilde m/p ^{k(p)}$.

Now observe that each ${\rm Tr}_{K(\theta _{i})/K}(\theta _{i})$ is an algebraic integer. To see this, assume that $\theta_i^{n_i}=1$, and let $p_i(x)\in K[x]$ be the minimal polynomial of $\theta_i$ over $K$. According to \cite[Proposition~1, p.~44]{Sa} (or to \cite[Theorem~8.2]{Morandi}), one may write that $${\rm Tr}_{K(\theta _{i})/K}(\theta _{i})=x_{i1}+x_{i2}+\cdots+x_{is_i},$$ where each $x_{ij}$ is a root of $p_i(x)$ in some extension field of $K$. Since $p_i(x)$ divides $f_i(x)=x^{n_i}-1$, each $x_{ij}$ is also a root of $f_i(x)$ and, therefore, an algebraic integer. Consequently, each ${\rm Tr}_{K(\theta _{i})/K}(\theta _{i})$ is an algebraic integer as well. Thus, equation~(\ref{eq11}) implies that $\mu _{p}/p$ is an algebraic integer. As $p$ does not divide $\mu_p$, one sees that $\mu _{p}/p$ lies in the set complement $\mathbb{Q} \setminus \mathbb{Z}$. However, our conclusion contradicts \cite[Example~5.0]{Atiyah}, and so $D$ is of positive characteristic, as promised.
\end{proof}


We think that in Proposition~\ref{FC-field}, the positive characteristic cannot occur either.

\begin{problem}
Let $K$ be either a local or a global field, and let $D$ be an LBD division $K$-algebra. Can $D$ be additively periodic?
\end{problem}


We are closing this subsection with a result that is similar to the one known for periodic rings.

\begin{proposition}\label{nil-lift}
Assume that $R$ and $S$ are two rings, and let $I$ be a nil-ideal of $R$. Then,
\begin{itemize}
\item [(1)] $R\times S$ is additively periodic  if, and only if, $R$ and $S$ are additively periodic.\footnote{Note that it is different from the assertion that the direct product of two periodic rings is again a periodic ring (cf. \cite[Remark~3.5]{Cui}).}
\item[(2)] If $\ch (R)>0$ and $R/I$ is additively  $k$-periodic, then $R$ is additively $(k+1)$-periodic.
\end{itemize}
\end{proposition}

\begin{proof}
(1) We demonstrate that if $x \in R$ and $y \in S$ are periodic, then $(x,y) \in R \times S$ is also periodic. Suppose that $x^k=x^l$ ($k>l$), and $y^m=y^n$ ($m>n$). Therefore, $$x^l=x^k=x^{k-l+l}=x^{2(k-l)+l}=\dots =x^{r(k-l)+l},$$ and
$$y^n=y^m=y^{m-n+n}=y^{2(m-n)+n}=\dots=y^{r(m-n)+n}$$ for any positive integer $r$. Let $s=(k-l)(m-n)$ and $t={\rm max}\{l,n\}$. Thus, it is easy to see that $(x,y)^{s+t}= (x,y)^t$, as desired.

Now, by the observation in the previous paragraph, it is evident that if $R$ and $S$ are additively periodic, so is the direct product $R\times S$. The converse statement is trivially true.

\medskip

(2) We first assert that:

\medskip

\noindent{{\bf Claim.}} {\it If $a+I$ is a periodic element in $R/I$, then $a$ is periodic itself.}

\medskip

(Notice that by \cite[Corollary~3.7]{Cui}, if $R/I$ is a periodic ring, then so is $R$, but this does {\it not} imply our assertion.)

\medskip

To see what we claimed, firstly suppose that $\ch (R)=p$, a prime. Now, there exist two different natural numbers $m$ and $n$ such that $a^m-a^n\in I$. Assume $(a^m-a^n)^s=0$ for some non-negative integer $s$, and choose $l\in \N$ such that $p^l\geq s$. So,
$$0=(a^m-a^n)^{p^l}=a^{mp^l}-a^{np^l},$$
showing that $a$ is a periodic element.

Next, suppose that $\ch(R)=p^n$, for a prime $p$ and a natural $n\in \N$. Putting $J=I+(pR)$, we see that $J$ is a nil-ideal of $R$ and $a+J\in R/J$ is periodic. Since $\ch(R/J)=p$, it follows from the previous part that $a$ is periodic.

Now, assume the general case and let us write $\ch(R)=p_1^{n_1}p_2^{n_2}\dots p_k^{n_k}$, where all $p_i$'s are distinct primes. Thus, as it is well-known, we can write $R\simeq \prod_{j=1}^k R_j$, where, for each $j$, $R_j\simeq R/(p_j^{n_j} R)$ is a ring of characteristic $p_j^{n_j}$. We know that $I=\prod_{j=1}^k I_j$, where, for each $j$, $I_j$ is a nil-ideal of $R_j$. Write $$a=(a_1,a_2,\ldots,a_k),$$ where $a_j\in R_j$. Then, one sees that $a_j+I_j$ is a periodic element of $R_j/I_j$ for all $j$. According to the previous paragraph, $a_j$ is a periodic element in $R_j$ for all $j$. This implies by part~(1) that $a$ is a periodic element of $R$, and the claim is now established.

\medskip

Now, suppose that  $R/I$ is additively $k$-periodic  and let $a\in R$. Hence
$$a+I=\sum_{j=1}^k a_j+I,$$
where $a_j+I$ is a periodic element of $R/I$ for each $j$. Therefore,
$$a=(\sum_{j=1}^k a_j)+b$$
for some $b\in I$. Referring to the above Claim, for each $j$, the element $a_j$ is periodic and, clearly, $b$ is also periodic. This shows that $R$ is additively $(k+1)$-periodic, as required.
\end{proof}

\medskip


\medskip

\subsection{Group rings}

As a direct consequence of Theorem~\ref{algebraic}, it follows that if $F$ is a field and $G$ is a locally finite group such that the group algebra $FG$ is additively $2$-periodic, then $FG$ is  periodic itself. Now, as a  consequence of Theorem~\ref{commuting}, we can state and prove the following more general result.

\begin{corollary}\label{cor2}
Let $R$ be a commutative ring and let $G$ be a locally finite group. If  the group ring $RG$ is additively periodic, then $RG$ is periodic.
\end{corollary}

\begin{proof}
We know that $$RG/w(RG)\simeq R,$$ where $w(RG)$ is the  augmentation ideal of $RG$. Therefore, $R$ is additively periodic. As $R$ is commutative, Theorem~\ref{commuting} ensures that $R$ is a periodic ring. Now, according to \cite[Corollary 1.3]{ABD}, $RG$ is a periodic ring; however, since \cite{ABD} has not been published yet, we will provide a sketch of the proof here in order to convince the reader: in fact, observe that the quotient-ring $R/{\rm Nil}(R)$ is locally finite by \cite[Corollary 2]{Hirano}. Then, owing to \cite[Proposition 2.12]{Kim}, the ring $$RG/{\rm Nil(R)}G\simeq (R/{\rm Nil}(R))G$$ is locally finite, too. However, since $R$ is a commutative ring, it is easy to see that ${\rm Nil}(R)G$ is a nil-ideal of $RG$. Consequently, \cite[Corollary~3.7]{Cui} assures that $RG$ is periodic, as asserted.
\end{proof}


Our main result on this occasion states the following.

\begin{theorem}\label{nilpotent}
Let $R$ be a commutative ring and let $G$ be a nilpotent group. Then, the following three statements are equivalent:
\begin{itemize}
\item [(1)] $RG$ is additively $2$-periodic.
\item [(2)] $RG$ is periodic.
\item [(3)] $R$ is periodic and $G$ is locally finite.
\end{itemize}
\end{theorem}

\begin{proof}
The verification of (3) implying (2) follows from \cite[Corollary 1.3]{ABD}, and it is evident that (2) implies (1). Therefore, it suffices to prove only the implication (1)~$\Rightarrow$~(3). To that purpose, suppose that $RG$ is additively $2$-periodic. Based on what we observed in the proof of Corollary~\ref{cor2}, $R$ is periodic.

Next, we claim that $G$ is locally finite. Let $G$ be nilpotent of class $n$. We shall prove our claim by induction on $n$. For $n=1$, the group $G$ is Abelian. Thus, invoking Theorem~\ref{commuting}, $RG$ is periodic. Therefore, by \cite[Proposition~3.4]{Chin}, the group $G$ is torsion (= locally finite).

Now, assume that the claim is true for all nilpotent groups of class less than $n$. Note that $G/{\rm Z}(G)$ is a nilpotent group of class $n-1$. Additionally, each element of $$R(G/{\rm Z}(G))\cong RG/[RG.w(R({\rm Z}(G)))],$$ as being a homomorphic image of $RG$ (cf. \cite[Corollary~3.3.5]{Sehgal}), can be expressed as a sum of two periodic elements. Therefore, by hypothesis, the factor-group $G/{\rm Z}(G)$ is locally finite.

Furthermore, we intend to show that $R{\rm Z}(G)$ is periodic. In fact, if $\alpha\in R{\rm Z}(G)$, then there are two elements $\beta_1$ and $\beta_2$ in $RG$ such that $\alpha=\beta_1+\beta_2$. As $\alpha$ is central in $RG$, one inspects that $\beta_1$ commutes with $\beta_2$. Therefore, $\alpha$ is periodic by Lemma~\ref{commute-periodic}. Hence, $R{\rm Z}(G)$ is periodic, as well, and so ${\rm Z}(G)$ is locally finite.

Finally, by what we have established so far, both of the two groups ${\rm Z}(G)$ and $G/{\rm Z}(G)$ are locally finite. Consequently, according to a result due to Schmidt (see, e.g., \cite[Statement 14.3.1]{robinson}), the group $G$ is also locally finite, as desired.
\end{proof}

\medskip


\medskip

\subsection{Triangular and full matrix rings}

As usual, the symbol $\mathbb{T}_n(R)$ is reserved for the (upper or lower) {\it triangular matrix ring} of size $n$, where $n\in \mathbb{N}$.

\medskip

Our pivotal achievement here states as follows.

\begin{theorem}\label{triang}
Suppose $R$ is a ring and $n\geq 1$.
\begin{itemize}
\item [(1)] If $R$ is additively periodic, then so is ${\rm M}_n(R)$.
\item [(2)]  Let $R$ be commutative. Then ${\rm M}_n(R)$ is additively $2$-periodic if, and only if, ${\rm M}_n(R)$ is periodic if, and only if, $R$ is periodic.
\item [(3)] The ring $\mathbb{T}_n(R)$ is additively periodic  if, and only if, $R$ is additively  periodic.
\item[(4)] Let $R$ be commutative. Then $\mathbb{T}_n(R)$ is additively  periodic if, and only if, $\mathbb{T}_n(R)$ is periodic if, and only if, $R$ is periodic.
\end{itemize}
\end{theorem}

\begin{proof}
To prove assertion (1),  assume that $R$ is additively periodic. Let $(a_{ij})\in {\rm M}_n(R)$. We have  $$(a_{ij})=\sum_{1\le i, j\le n}a_{ij}e_{ij},$$ where $e_{ij}$ is the elementary matrix in which the $(i,j)$-entry is $1$ and all other entries are $0$. Putting
$$A=\sum_{1\le i<j\le n}a_{ij}e_{ij}, ~~~B=\sum_{1\le j<i\le n}a_{ij}e_{ij}~ \text{~and~}~ C=\sum_{1\le i\le n}a_{ii}e_{ii},$$
we see that $(a_{ij})=A+B+C$.

Observe that, $A$ and $B$ are nilpotent, so they are periodic. Additionally, since $R$ is additively periodic,  one has
$$a_{ii}=\sum_{\ell=1}^k {b_{i\ell}},$$
where $k$ is a natural number and $b_{i\ell} \in R$ is periodic  for all $1\le i\le n$ and  $1\leq\ell\leq k$. Now, we can write
$$C=\sum_{\ell=1}^k\left(\sum_{i=1}^n {b_{i\ell}}e_{ii}\right).$$

It is easy to see that for each $i$ and $\ell$, the element ${b_{i\ell}}e_{ii}$ is periodic. Further, by Proposition~\ref{nil-lift}(1), the diagonal matrix $\sum_{i=1}^n {b_{i\ell}}e_{ii}$ is periodic. Therefore, $C$ is also a periodic element, and hence $(a_{ij})$ is a sum of periodic elements, as required.

Let $R$ be commutative and let ${\rm M}_n(R)$ be additively $2$-periodic. If $x\in R$, then, considering $x$ as an element in ${\rm M}_n(R)$, it can be expressed as the sum of two periodic elements $\alpha$ and $\beta$ in ${\rm M}_n(R)$. Additionally, since $x$ is central, $\alpha\beta=\beta\alpha$. As a result, referring to Lemmas~\ref{char} and \ref{commute-periodic}, $x$ is periodic. Consequently, $R$ is a periodic ring, implying that ${\rm M}_n(R)$ is also periodic  (see, e.g.,  \cite[Corollary~2.5]{ABD}). This proves (2).

For assertion (3), we know that there is a nil-ideal $I$ of $\mathbb{T}_n(R)$ such that $$\mathbb{T}_n(R)/I\simeq R\times \dots \times R,$$ where the direct product is taken $n$ times. This shows that if $\mathbb{T}_n(R)$ is additively periodic, then so does $\mathbb{T}_n(R)/I$ and thus, $R$ is additively periodic as being an epimorphic image. The converse follows directly from (1).

Finally, assertion (4) follows from combination of point (3), Theorem~\ref{commuting}, \cite[Remark~3.5 and Corollary~3.7]{Cui}.
\end{proof}


\begin{remark}
At this stage, we are unready to decide whether or not the converse of Theorem~\ref{triang} (1) is true and, if yes, the proof seems to be technically difficult.

Besides, it is still unknown whether or not ${\rm M}_n(R)$ is periodic provided that $R$ is periodic (see, e.g., \cite{ABD}). However, as the proof of Theorem~\ref{triang} shows, if $R$ is periodic, then each $(a_{ij})\in{\rm M}_n(R)$ is a sum of one periodic element, $\diag(a_{11},a_{22},\ldots, a_{nn})$, and at most $2$ nilpotents (whence, $(a_{ij})$ is a sum of at most $3$ periodic elements). More precisely,  if $R$  is additively $k$-periodic, then ${\rm M}_n(R)$ and $\mathbb{T}_n(R)$  are additively  $(k+2)$-periodic and $(k+1)$-periodic, respectively.
\end{remark}

\medskip


\medskip

\subsection{The torsion product property}

In this subsection, we investigate the rings with the torsion product property, where every element is a sum of a finite number of  torsion elements.

For any ring $R$, we denote by $\mathcal{TU}(R)$ the set of elements of finite order in $\mathcal{U}(R)$. We shall say that $R$ has the {\it torsion product property} (or, briefly, that $R$  has {\it t.p.p}) whenever $\mathcal{TU}(R)$ is a subgroup of $\mathcal{U}(R)$. The torsion product property has been studied in various important contexts -- e.g., for associative division rings and matrix rings over division rings, we refer to \cite{CPM}; in group rings, we refer to \cite{Bist, Coel}; and in alternative loop algebras over fields, we refer to \cite{Good}.

We say that the ring $R$ is {\it additively torsion} if every element of $R$ is a sum of a finite number of torsion elements, and {\it additively $2$-torsion} refers to the ring $R$ in which each element is a sum of two torsion elements.

\medskip
We begin with the following useful technical claim.

\begin{lemma}\label{t.p.p}
Let the ring $R$ have t.p.p and $\ch (R)=p$, a prime number. If $R$ is additively torsion, then $R$ is a locally finite field.
\end{lemma}

\begin{proof} Choose any $0\neq \alpha\in R$. Assume that there are two torsion elements $a$ and $b$ in $\mathcal{TU}(R)$ such that $\alpha=a+b$. Since $R$ has t.p.p, it must be that $b^{-1}a \in \mathcal{TU}(R)$. As $\mathbb F_p(a^{-1}b)$ is a finite field, there exists a positive integer $m$ such that $$(a^{-1}b)^{p^m}=a^{-1}b,$$ and hence $$(1+a^{-1}b)^{p^m}= 1+a^{-1}b.$$ This shows that $1+a^{-1}b$ is a torsion element. Accordingly, $$\alpha=a( 1+a^{-1}b)$$ is also torsion. We, thus, have shown that the sum of two torsion elements is again a torsion element. Therefore, by induction, we can deduce that every non-zero element of $R$ is a torsion unit. Consequently, $R$ is a locally finite field, as claimed.
\end{proof}


We say that the ring $R$ has {\it strongly t.p.p} if, for each nil-ideal $I$ of $R$, the factor ring $R/I$ has t.p.p. For example, if $\mathcal{U}(R)$ is a locally nilpotent group, then $R$ has strongly t.p.p: For if $I$ is a nil-ideal, then $\mathcal{U}(R/I)$ is also a locally nilpotent group, hence its torsion elements $\mathcal{TU}(R/I)$ form a subgroup of $\mathcal{U}(R/I)$.

\medskip

Our main result in this subsection is as follows.

\begin{theorem}\label{strongly}
Let the ring $R$ have strongly t.p.p. Then, in each of the following cases, $R$ is a periodic ring:
\begin{itemize}
  \item [(1)] If $\text{char}(R) > 0$ and $R$  is additively torsion.
  \item [(2)] If  $R$  is additively $2$-torsion.
\end{itemize}
\end{theorem}

\begin{proof}
Let us write $\ch(R)=p_1^{n_1}p_2^{n_2}\dots p_k^{n_k}$, where all $p_i$'s are distinct primes. Thus, we have the decomposition $$R\simeq \prod_{i=1}^k R_i,$$ where, for each $i$, $R_i$ is a ring of characteristic $p_i^{n_i}$. According to \cite[Remark~3.5]{Cui}, it suffices to show that, for every $i$, $R_i$ is a periodic ring. As it can be easily seen, every $R_i$ has strongly t.p.p such that each of which element is a sum of a finite number of torsion elements in $R_i$, so we may assume with no harm in generality that $R$ itself has a prime power characteristic, say $p^n$.

Now, as $pR$ is obviously a nil-ideal of $R$, it follows that $R/(pR)$ has t.p.p. Therefore, Lemma~\ref{t.p.p} implies that $R/(pR)$ is a periodic ring. Hence, according to \cite[Corollary~3.7]{Cui}, we deduce that $R$ is periodic as well, thus proving (1).

Furthermore, one observes that the assertion (2) is a special case of (1) since, as stated in Lemma~\ref{char}, $R$ has a positive characteristic.
\end{proof}


As two immediate consequences of Theorem~\ref{strongly}, we extract:

\begin{corollary}\label{2-torsion}
Suppose that the ring $R$ is additively $2$-torsion. If $\mathcal{U}(R)$ is either a torsion or a locally nilpotent group, then $R$ is periodic.
\end{corollary}


Recall that a ring $R$ is said to be $2$-{\it good} if every element of $R$ is a sum of two units (see \cite{2-good2, 2-good}).

\begin{corollary}\label{2-good}
Suppose that the ring $R$ is a $2$-good ring. If $\mathcal{U}(R)$ is  a torsion group, then $R$ is periodic.
\end{corollary}

We close this section by noticing that, in connection with Theorem~\ref{algebraic}, Corollary~\ref{2-torsion} and Corollary~\ref{2-good}, it may come to mind that if $\mathcal{U}(R)$ is a torsion group, then $R$ is always periodic. However, this is manifestly {\it untrue}. As a simple example, let $F$ be a finite field. Then, the unit group of the polynomial ring $F[x]$ is too finite, while $F[x]$ is {\it not} a periodic ring. In fact, for each ring $R$, the Laurent polynomial ring $R[x, x^{-1}]$, the power series ring $R[[x]]$, and the polynomial ring $R[x]$ are {\it not} additively $2$-periodic. Otherwise, the central element $x$ would be periodic by virtue of Lemmas~\ref{char} and \ref{commute-periodic}, which would lead to a contradiction. However, we do not know whether or {\it not} these rings are additively periodic.

\medskip


\medskip

\section{Concluding Discussion and Open Problems}

We close the work with the following four intriguing queries. Firstly, we inquire whether the converse of Theorem~\ref{triang}(1) is true.

\begin{problem}
Suppose that $R$ is a ring and $n\geq 2$ is an integer. Does it follow that the matrix ring ${\rm M}_n(R)$ being additively periodic will imply the same property for $R$?
\end{problem}

We conjecture that the answer is {\it no}.

\medskip


Our second question is a main tool for further studies over rings additively generated by periodic elements.

\begin{problem}
Is any additively $2$-periodic ring a (weakly) periodic ring or {\it not}?
\end{problem}


Certainly, all finite rings are periodic, but not all of them are potent rings however.

\medskip

Now, we ask the following:

\begin{problem}
Is any element of a finite ring a sum of (at least) two potents and, if {\it not}, decide when it is possible?
\end{problem}


The next challenging problem arisen from the main results obtained in \cite{ACDT} and \cite{YKZ}, respectively (see also \cite[Question 6.4]{YKZ}).

\begin{problem}
Examine those rings in which every element is a sum of two (concrete) potents that commute with each other, especially either the sum of commuting $3$-potent and $4$-potent, or the sum of commuting $3$-potent and $5$-potent, respectively.
\end{problem}


\noindent{\bf Acknowledgment.} The authors are deeply thankful to Prof. Ivan D. Chipchakov from the Institute of Mathematics \& Informatics of the Bulgarian Academy of Sciences for the valuable discussion on the initial version the paper, which led to an improvement in its quality.


\medskip

\noindent{\bf Declarations.} Our statements here are the following:
\begin{itemize}
\item {\bf Ethical Declarations and Approval:} The authors have no any competing interest to declare that are relevant to the content of this article.
\item {\bf Competing Interests:} The authors declare no any conflict of interest.
\item  {\bf Authors' Contributions:} All three listed authors worked and contributed to the paper equally. The final editing was done by the corresponding author P.V. Danchev and was approved by all of the present authors.
\item {\bf Availability of Data and Materials:} Data sharing not applicable to this article as no data-sets or any other materials were generated or analyzed during the current study.
\end{itemize}


\medskip

\noindent{\bf Funding.} The research work of the first-named author (M. H. Bien) is funded by University of Science, VNUHCM under Grant Nos. T2023-04 and T2023-05. The research work of the second-named author (P.V. Danchev) is supported in part by the Bulgarian National Science Fund under Grant KP-06 No. 32/1 of December 07, 2019, as well as by the Junta de Andaluc\'ia under Grant FQM 264, and by the BIDEB 2221 of T\"UB\'ITAK. The research work of the third-named author (M. Ramezan-Nassab) is supported in part by a grant from IPM (Grant No. 1402160023).


\end{document}